\documentclass[12pt]{article}
\usepackage[a4paper,margin=1.2in]{geometry}
\usepackage{amssymb}
\usepackage{amsmath}
\usepackage{amsfonts}
\usepackage{amsthm}
\usepackage{color}
\usepackage{float}
\usepackage[all]{xy}
\usepackage{scalerel}
\usepackage{mathabx}
\usepackage{delimset}
\usepackage{mathtools}

\newcommand{\qbinom}{\genfrac{[}{]}{0pt}{}}

\newcommand{\C}{\mathbb{C}}

\newcommand{\N}{\mathbb{N}}

\newcommand{\E}{\mathrm{E}}
\newcommand{\e}{\mathrm{e}}
\newcommand{\h}{\mathrm{h}}
\newcommand{\s}{\mathrm{S}}

\newcommand{\M}{\widehat{\mathrm{M}}}
\newcommand{\I}{\widehat{\mathrm{I}}}
\newcommand{\D}{\mathbf{D}}

\newtheorem{theorem}{Theorem}

\newtheorem{definition}{Definition}
\newtheorem{proposition}{Proposition}

\newtheorem{example}{Example}

\begin{document}

\title{\textbf{Homogeneous Linear Calculus of Order 1 and a $\lambda$-Taylor Formula}}
\author{Ronald Orozco L\'opez}

\newcommand{\Addresses}{{
  \bigskip
  \footnotesize

  \textit{E-mail address}, R.~Orozco: \texttt{rj.orozco@uniandes.edu.co}
  
}}

\maketitle

\begin{abstract}
In this paper, a new calculus on sequences is defined. Also, the $\lambda$-derivative and the $\lambda$-integration are investigated. The fundamental theorem of $\lambda$-calculus is included. A suitable function basis for the $\lambda$-derivative and the $\lambda$-integral is provided, and various properties of this basis are given. A $\lambda$-Taylor formula for functions is given. 
\end{abstract}
\noindent 2020 {\it Mathematics Subject Classification}:
Primary 05A30. Secondary 33D45.

\noindent \emph{Keywords: } $\lambda$-calculus, $\lambda$-derivative, $\lambda$-integration, $\lambda$-Taylor theorem.

\section{Introduction}

Ward \cite{ward} introduced a calculus on a sequence $\{\psi_{n}\}_{n\in\N}$, with $\psi_{0}=0$ and $\psi_{n}\neq0$ for all $n\geq1$, in such a way that when $\psi_{n}=n$, we get the ordinary calculus, when $\psi_{n}=\frac{1-q^n}{1-q}$, is obtained the $q$-calculus \cite{jackson1,jackson2,jackson3}, and if $\psi_{n}=F_{n}$, where $\{F_{n}\}$ is the Fibonacci sequence, then we obtain the Golden calculus \cite{pashaev1,pashaev2}. A Ward-derivative $\D$ is a linear operator which satisfies  $\D(x^n)=\psi_{n}x^{n-1}$, for all $n\in\N$. The ring of constants associated with the derivative $\D$ is the set of complex numbers $\C$.

In this paper, a new calculus on sequences is introduced: the homogeneous linear calculus of order one or $\lambda$-calculus. This calculus is defined on the sequence $\psi_{n}=\lambda^n$, $n\geq0$, $\lambda\in\C$, which is the solution of the homogeneous linear difference equation of order 1
\begin{equation}\label{eqn_HLE}
    \psi_{n+1}-\lambda\psi_{n}=0.
\end{equation}
This new calculus is not Ward's calculus. We define the $\lambda$-derivative $\D_{\lambda}$ of the function $f(x)$ by $\D_{\lambda}f(x)=\frac{f(\lambda x)}{x}$. Then $\D_{\lambda}(x^{n})=\lambda^nx^{n-1}$, for all $n\geq0$, and thus $\D_{\lambda}(1)=\frac{1}{x}$ and $\D_{\lambda}(0)=0$. Therefore, the ring of constants associated with the derivative $\D_{\lambda}$ is the singleton set $\{0\}$. The $q$-derivative of Jackson $D_{q}f(x)=\frac{f(x)-f(qx)}{(1-q)x}$ is expressed as linear combination of $\lambda$-derivative, i.e., 
\begin{equation}
    D_{q}=\frac{1}{1-q}(\D_{1}-\D_{q}).
\end{equation}
Then, the $\lambda$-calculus is the simplest calculus, and we can use $\lambda$-derivatives to construct new derivatives. 

The paper is organized as follows. Section \ref{sec3} deals with the properties of the $\lambda$-derivative $\D_{\lambda}$. In this same section, we also study the properties of the $\lambda$-integral of the function $f(x)$ defined as $\I_{\lambda}f(x)=\frac{x}{\lambda}f(x/\lambda)$.

Section \ref{sec4} deals with the functions $\e_{1/\lambda}(a/x)$ and $\E_{1/\lambda}(a/x)$, where $\e_{q}(x)$ and $\E_{q}(x)$ are the $q$-exponential functions. In Section \ref{sec5} we introduce the functions
\begin{equation}
    (x-a)_{\lambda}^n=\left(1-\frac{a}{x}\right)\left(1-\frac{a}{\lambda x}\right)\cdots\left(1-\frac{a}{\lambda^{n-1}x}\right),
\end{equation}
for any positive integer $n\geq1$ y $(x-a)_{\lambda}^0=1$, which will serve as basis functions for our $\lambda$-Taylor theorem.. These functions are generalized for all $n\in\C$ in Section \ref{sec6}. In Section \ref{sec7} we provide a $\lambda$-Taylor theorem for functions $f(x^{-1})$ by using the basis functions $(x-a)_{\lambda}^n$ and the operator $\I_{\lambda}$, where $f(x)$ is a polynomial of degree $n$. Some connection formulas are given.

\section{Preliminaries from $q$-calculus} 

The $q$-shifted factorial be defined by
\begin{align*}
    (a;q)_{n}&=\begin{cases}
        1&\text{ if }n=0;\\
    \prod_{k=0}^{n-1}(1-q^{k}a),&\text{ if }n\neq0,\\
    \end{cases}\hspace{1cm}q\in\C,\\
    (a;q)_{\infty}&=\lim_{n\rightarrow\infty}(a;q)_{n}=\prod_{k=0}^{\infty}(1-aq^{k}),\hspace{1cm} \vert q\vert<1.
\end{align*}
The $q$-binomial coefficient is defined by
\begin{equation*}
\qbinom{n}{k}_{q}=\frac{(q;q)_{n}}{(q;q)_{k}(q;q)_{n-k}}.
\end{equation*}
The $q$-exponential $\e_{q}(z)$, \cite{euler,kac}, is defined by
\begin{equation*}
    \e_{q}(z)=\sum_{n=0}^{\infty}\frac{z^n}{(q;q)_{n}}=\frac{1}{(z;q)_{\infty}}.
\end{equation*}
Another $q$-analogue of the classical exponential function is
\begin{equation*}
    \E_{q}(z)=\sum_{n=0}^{\infty}q^{\binom{n}{2}}\frac{z^n}{(q;q)_{n}}=(-z;q)_{\infty}.
\end{equation*}
Some special $q$-polynomials: The $q$-binomial theorem
\begin{equation}\label{eqn_gauss}
    (x;q)_{n}=\sum_{k=0}^{n}\qbinom{n}{k}_{q}q^{\binom{k}{2}}(-1)^kx^k.
\end{equation}
The classical Rogers-Szeg\"o polynomials defined by
\begin{equation}\label{eqn_rs}
    \h_{n}(x|q)=\sum_{k=0}^{n}\qbinom{n}{k}_{q}x^k.
\end{equation}
The Stieljes-Wigert polynomials are defined by
\begin{equation}\label{eqn_sw}
    \s_{n}(x|q)=\frac{1}{(q;q)_{n}}\sum_{k=0}^{n}\qbinom{n}{k}_{q}q^{k^2}x^k.
\end{equation}

\section{Homogeneous linear calculus of order 1}\label{sec3}

\subsection{$\lambda$-differential operator}

\begin{definition}
We define the $\lambda$-derivative $\mathbf{D}_{\lambda}$ of the function $f(x)$ as
\begin{equation}
    (\mathbf{D}_{\lambda}f)(x)=\frac{1}{x}\M_{\lambda}\{f(x)\}=
    \begin{cases}
    \frac{f(\lambda x)}{x},&\text{ if }x\neq0;\\
    \lim_{x\rightarrow0}\frac{f(\lambda x)}{x},&\text{ if }x=0,
    \end{cases}
\end{equation}    
provided that the limit exists.
\end{definition}

It is straightforward to prove the following properties of the $\lambda$-differential operator $\mathbf{D}_{\lambda}$.
\begin{proposition}\label{prop_properties}
For all $\alpha,\beta,\gamma\in\C$,
\begin{enumerate}
    \item $\mathbf{D}_{\lambda}\{\alpha f+\beta g\}=\alpha\mathbf{D}_{\lambda}f+\beta\mathbf{D}_{\lambda}g$.
    \item $\mathbf{D}_{\lambda}\{\gamma\}=\frac{\gamma}{x}$.
    \item $\mathbf{D}_{\lambda}\{x^{n}\}=\lambda^{n}x^{n-1}$, for $n\in\N$.
\end{enumerate}
\end{proposition}
\begin{proposition}[{\bf Product $\lambda$-rule}]\label{prop_prod_rule}
\begin{equation}
 \mathbf{D}_{\lambda}\{f(x)g(x)\}=f(\lambda x)\mathbf{D}_{\lambda}g(x)=\mathbf{D}_{\lambda}f(x)\cdot g(\lambda x)=x\mathbf{D}_{\lambda}f(x)\mathbf{D}_{\lambda}g(x).   
\end{equation}
\end{proposition}
From the above proposition, we have the following result.
\begin{proposition}
    \begin{equation}
        \mathbf{D}_{\lambda}\{f(x)g(x)\}=\frac{1}{2}(f(\lambda x)\mathbf{D}_{\lambda}g(x)+\mathbf{D}_{\lambda}f(x)\cdot g(\lambda x))
    \end{equation}
\end{proposition}

\begin{proposition}[{\bf Quotient $\lambda$-rule}]
\begin{equation}
 \mathbf{D}_{\lambda}\left\{\frac{f(x)}{g(x)}\right\}=\frac{\mathbf{D}_{\lambda}f(x)}{g(\lambda x)}=\frac{\mathbf{D}_{\lambda}f(x)}{x\mathbf{D}_{\lambda}g(x)}, 
\end{equation}
$g(x)\neq0$, $g(\lambda x)\neq0$.
\end{proposition}

\begin{proposition}\label{prop_dnf}
For all $n\in\N$
    \begin{equation}
        \mathbf{D}_{\lambda}^{n}f(x)=\frac{f(\lambda^nx)}{\lambda^{\binom{n}{2}}x^n}.
    \end{equation}
\end{proposition}
\begin{proof}
The proof is by induction on $n$. For $n=1$, use the definition $\mathbf{D}_{\lambda}$. Now suppose it is true for $n$. We will prove for $n+1$. We have
\begin{align*}
    \mathbf{D}_{\lambda}^{n+1}f(x)=\mathbf{D}_{\lambda}\mathbf{D}_{\lambda}^{n}f(x)=\mathbf{D}_{\lambda}\left\{\frac{f(\lambda^nx)}{\lambda^{\binom{n}{2}}x^{n}}\right\}=\frac{f(\lambda^{n}\lambda x)}{\lambda^{\binom{n}{2}}\lambda^nx^nx}=\frac{f(\lambda^{n+1}x)}{\lambda^{\binom{n+1}{2}}x^{n+1}}
\end{align*}
\end{proof}

\begin{example}
For all $\alpha\in\C$ and all $k\geq1$,
\begin{equation}
    \D_{\lambda}^kx^\alpha=\lambda^{k\alpha-\binom{k}{2}}x^{\alpha-k}.
\end{equation}
\end{example}

\begin{proposition}[{\bf Leibniz $\lambda$-rule}]
\begin{equation}
    \mathbf{D}_{\lambda}^{n}(fg)=\lambda^{\binom{n}{2}}x^{n}\mathbf{D}_{\lambda}^{n}(f)\mathbf{D}_{\lambda}^{n}(g).
\end{equation}    
\end{proposition}
\begin{proof}
From Proposition \ref{prop_properties} and \ref{prop_dnf}
\begin{align*}
    \mathbf{D}_{\lambda}^{n}(fg)&=\frac{f(\lambda^{n}x)g(\lambda^{n}x)}{\lambda^{\binom{n}{2}}x^n}\\
    &=\lambda^{\binom{n}{2}}x^n\frac{f(\lambda^nx)}{\lambda^{\binom{n}{2}}x^n}\frac{g(\lambda^nx)}{\lambda^{\binom{n}{2}}x^n}\\
    &=\lambda^{\binom{n}{2}}x^{n}\mathbf{D}_{\lambda}^n(f)\mathbf{D}_{\lambda}^{n}(g).
\end{align*}    
\end{proof}

\subsection{$\lambda$-integral operator}

\begin{definition}
Let $f:[0,x]\rightarrow\C$ be a continuous function. The $\lambda$-integral of the function $f(x)$ in $[0,x]$ is defined as
    \begin{equation}\label{eqn_int1}
        \I_{\lambda}\{f(x)\}=\int_{0}^{x} f(t)d_\lambda t=
        \M_{1/\lambda}\{xf(x)\}=\frac{x}{\lambda}f\left(\frac{x}{\lambda}\right)
    \end{equation}
provided that $f(x/\lambda)$ exists and define, also
\begin{equation}\label{eqn_int2}
    \int_{\alpha}^{\beta}f(x)d_{\lambda}x=\int_{0}^{\beta}f(x)d_{\lambda}x-\int_{0}^{\alpha}f(x)d_{\lambda}x.
\end{equation}
We will say that $f$ is $\lambda$-integrable in $[\alpha,\beta]$ provided that $f$ is defined in $x=\alpha/\lambda$ and $x=\beta/\lambda$.
\end{definition}
From Eqs. (\ref{eqn_int1}) and (\ref{eqn_int2}) it follows that
\begin{equation}\label{eqn_int_int}
    \int_{\alpha}^{\beta}f(x)d_{\lambda}x=\frac{\beta}{\lambda}f\left(\frac{\beta}{\lambda}\right)-\frac{\alpha}{\lambda}f\left(\frac{\alpha}{\lambda}\right).
\end{equation}

\begin{proposition}
For all $\alpha,\beta,\gamma\in\C$
\begin{enumerate}
    \item $\int_{\alpha}^{\beta}[f(x)+g(x)]d_{\lambda,a}x=\int_{\alpha}^{\beta}f(x)d_{\lambda,a}x+\int_{\alpha}^{\beta}f(x)d_{\lambda,a}x$.
    \item $\int_{\alpha}^{\alpha}f(x)d_{\lambda,a}x=0$.
    \item $\int_{\alpha}^{\beta}f(x)d_{\lambda,a}x=-\int_{\beta}^{\alpha}f(x)d_{\lambda,a}x$.
    \item $\int_{\alpha}^{\beta}f(x)d_{\lambda,a}x=\int_{\alpha}^{\gamma}f(x)d_{\lambda,a}x+\int_{\gamma}^{\beta}f(x)d_{\lambda,a}x$.
\end{enumerate}
\end{proposition}

\begin{proposition}\label{prop_nth-int}
For all $n\in\N$,
    \begin{equation}
        \I_{\lambda}^{n}\{f(x)\}=\frac{x^n}{\lambda^{\binom{n+1}{2}}}f\left(\frac{x}{\lambda^n}\right).
    \end{equation}
\end{proposition}
\begin{proof}
The proof is by induction on $n$. For $n=1$, we use the Eq.(\ref{eqn_int1}). Now, suppose the statement is true for $n$; we will prove it for $n+1$. We have 
\begin{align*}
    \I_{\lambda}^{n+1}\{f(x)\}&=\I_{\lambda}\left\{\I_{\lambda}^{n}\{f(x)\}\right\}\\
    &=\I_{\lambda}\left\{\frac{x^{n}}{\lambda^{\binom{n+1}{2}}}f\left(\frac{x}{\lambda^n}\right)\right\}\\
    &=\M_{1/\lambda}\left\{\frac{x^{n+1}}{\lambda^{\binom{n+1}{2}}}f\left(\frac{x}{\lambda^n}\right)\right\}\\
    &=\frac{x^{n+1}}{\lambda^{\binom{n+2}{2}}}f\left(\frac{x}{\lambda^{n+1}}\right)
\end{align*}
\end{proof}

\begin{example}
For all $\alpha\in\C$ and all $k\geq1$,
    \begin{equation}
        \I_{\lambda}^{k}x^{\alpha}=\frac{x^{k+\alpha}}{\lambda^{k\alpha+\binom{k+1}{2}}}
    \end{equation}
\end{example}

\begin{definition}
For any real number $p\geq1$ we will denote by $\mathcal{L}_{\lambda}^{p}[\alpha,\beta]$ the set of functions $f:[\alpha,\beta]\rightarrow\C$ such that $\vert f\vert^p$ is $\lambda$-integrable in $[\alpha,\beta]$, i.e.,
\begin{equation}
    \mathcal{L}_{\lambda}^{p}[\alpha,\beta]=\bigg\{f:[\alpha,\beta]\rightarrow\C\ \Big\vert\int_{\alpha}^{\beta}\vert f\vert^pd_{\lambda}<\infty\bigg\}.
\end{equation}
We also set
\begin{equation}
    \mathcal{L}_{\lambda}^{\infty}[\alpha,\beta]=\bigg\{f:[\alpha,\beta]\rightarrow\C\ \Big\vert\ \sup\Big\{\Big\vert f\left(\frac{\alpha}{\lambda}\right)\Big\vert,\Big\vert f\left(\frac{\beta}{\lambda}\right)\Big\vert\Big\}<\infty\bigg\}.
\end{equation}
\end{definition}

\begin{theorem}[{\bf Fundamental theorems of $\lambda$-calculus}]\label{theo_funda}
\begin{enumerate}
    \item Let $f:[\alpha,\beta]\rightarrow\C$ be a function such that $f\in\mathcal{L}_{\lambda}^{1}[\alpha,x]$ for all $x\in[\alpha,\beta]$. Then
    \begin{equation}
        \mathbf{D}_{\lambda}\int_{\alpha}^{x}f(t)d_{\lambda}t=f(x)-\frac{\alpha}{\lambda}f\left(\frac{\alpha}{\lambda}\right)\frac{1}{x}.
    \end{equation}
    \item Let $f:[\alpha,\beta]\rightarrow\C$ be a function such that $\mathbf{D}_{\lambda}f\in\mathcal{L}_{\lambda}^{1}[\alpha,\beta]$. Then
    \begin{equation}
        \int_{\alpha}^{\beta}\mathbf{D}_{\lambda}f(x)d_{\lambda}x=f(\beta)-f(\alpha).
    \end{equation}
\end{enumerate}
\end{theorem}
\begin{proof}
From Eq.(\ref{eqn_int_int}) and Proposition \ref{prop_prod_rule},
\begin{align*}
    \mathbf{D}_{\lambda}\int_{\alpha}^{x}f(t)d_{\lambda}t&=\mathbf{D}_{\lambda}\left(\frac{x}{\lambda}f\left(\frac{x}{\lambda}\right)-\frac{\alpha}{\lambda}f\left(\frac{\alpha}{\lambda}\right)\right)\\
    &=x\mathbf{D}_{\lambda}\left(\frac{x}{\lambda}\right)\mathbf{D}_{\lambda}f\left(\frac{x}{\lambda}\right)-\frac{\alpha}{\lambda}f\left(\frac{\alpha}{\lambda}\right)\mathbf{D}_{\lambda}(1)\\
    &=x\frac{\lambda}{\lambda}\frac{f(x)}{x}-\frac{\alpha}{\lambda}f\left(\frac{\alpha}{\lambda}\right)\frac{1}{x}\\
    &=f(x)-\frac{\alpha}{\lambda}f\left(\frac{\alpha}{\lambda}\right)\frac{1}{x}.
\end{align*}
and
\begin{align*}
    \int_{\alpha}^{\beta}\mathbf{D}_{\lambda}f(x)d_{\lambda}x&=\frac{\beta}{\lambda}(\mathbf{D}_{\lambda}f)\left(\frac{\beta}{\lambda}\right)-\frac{\alpha}{\lambda}(\mathbf{D}_{\lambda}f)\left(\frac{\alpha}{\lambda}\right)\\
    &=\frac{\beta}{\lambda}\frac{f(\beta)}{\beta/\lambda}-\frac{\alpha}{\lambda}\frac{f(\alpha)}{\alpha/\lambda}\\
    &=f(\beta)-f(\alpha).
\end{align*}
\end{proof}


\section{Solution of some proportional equations}\label{sec4}

\begin{theorem}
Suppose that $\lambda>1$.The solution of the proportional equation 
\begin{equation}\label{eqn_func1}
    xf(\lambda x)=xf(x)+af(\lambda x),\ f(\infty)=1,
\end{equation}
is
\begin{equation}
    f(x)=(a/x;\lambda^{-1})_{\infty}=\E_{1/\lambda}(-a/x).
\end{equation}
\end{theorem}
\begin{proof}
From the functional equation Eq.(\ref{eqn_func1})
\begin{equation}\label{eqn_sol_tay}
    f(x)=\left(1-\frac{a}{x}\right)f(\lambda x).
\end{equation}
Then, by iterating Eq.(\ref{eqn_sol_tay})
\begin{equation}
    f(x)=f(\lambda^{n}x)\prod_{k=0}^{n-1}\left(1-\frac{a}{\lambda^{k}x}\right).
\end{equation}
Then, by taking the limit as $n\rightarrow\infty$,
\begin{align*}
    f(x)&=\lim_{n\rightarrow\infty}f(\lambda^nx)\prod_{k=0}^{n-1}\left(1-\frac{a}{\lambda^kx}\right)\\
    &=f(\infty)\prod_{k=0}^{\infty}\left(1-\frac{a}{\lambda^kx}\right)\\
    &=\prod_{k=0}^{\infty}\left(1-\frac{a}{\lambda^kx}\right)=(a/x;\lambda^{-1}).
\end{align*}
\end{proof}

\begin{theorem}
Suppose that $\lambda>1$.The solution of the proportional equation 
\begin{equation}\label{eqn_func2}
    af(x)=xf(x)-xf(\lambda x),\ f(\infty)=1,
\end{equation}
is
\begin{equation}
    f(x)=\frac{1}{(a/x;\lambda^{-1})_{\infty}}=\e_{1/\lambda}(a/x).
\end{equation}
\end{theorem}
Eqs. (\ref{eqn_func1}) and (\ref{eqn_func2}) can be writing as
\begin{align*}
    a\D_{\lambda}f(x)&=f(\lambda x)-f(x),\\
    \D_{\lambda}f(x)&=\frac{1}{x}\left(1-\frac{a}{x}\right)f(x),
\end{align*}
respectively.

\begin{proposition}
For all $n\in\N$ and $\vert\lambda\vert<1$,
    \begin{align*}
        \D_{\lambda}^{n}(ax;\lambda)_{\infty}&=\frac{(ax;\lambda)_{\infty}}{\lambda^{\binom{n}{2}}x^n(ax;\lambda)_{n}}.\\
        \D_{\lambda}^{n}\frac{1}{(ax;\lambda)_{\infty}}&=\frac{(ax;\lambda)_{n}}{\lambda^{\binom{n}{2}}x^n(ax;\lambda)_{\infty}}.
    \end{align*}
\end{proposition}

\begin{proposition}
For all $n\in\N$ and $\vert\lambda\vert<1$,
    \begin{align*}
        \I_{\lambda}^{n}(ax;\lambda)_{\infty}&=\frac{x^n}{\lambda^{\binom{n+1}{2}}}(\lambda^{-1}ax;\lambda^{-1})_{n}(ax;\lambda)_{\infty}.\\
        \I_{\lambda}^{n}\frac{1}{(ax;\lambda)_{\infty}}&=\frac{x^n}{\lambda^{\binom{n+1}{2}}(\lambda^{-1}ax;\lambda^{-1})_{n}(ax;\lambda)_{\infty}}.
    \end{align*}
\end{proposition}

\section{$\lambda$-derivative and $\lambda$-integral of $(x-a)_{\lambda}^n$}\label{sec5}

\begin{definition}
The $\lambda$-analogue of $(x-a)^n$ is the function
    \begin{equation}
        (x-a)_{\lambda}^{n}=
        \begin{cases}
            1,&\text{ if }n=0;\\
            \prod_{k=0}^{n-1}\left(1-\frac{a}{\lambda^kx}\right),&\text{ if }n\neq1.
        \end{cases}
    \end{equation}
\end{definition}

\begin{proposition}
For all $n,k\in\N$,
    \begin{equation}
        \mathbf{D}_{\lambda}^k(x-a)_{\lambda}^n=\frac{1}{\lambda^{\binom{k}{2}}x^k}\frac{(x-a)_{\lambda}^{n+k}}{(x-a)_{\lambda}^k}.
    \end{equation}
\end{proposition}
\begin{proof}
    \begin{align*}
        \mathbf{D}_{\lambda}^{k}(x-a)_{\lambda}^n&=\frac{1}{\lambda^{\binom{k}{2}}x^k}(\lambda^k x-a)_{\lambda}^n\\
        &=\frac{1}{\lambda^{\binom{k}{2}}x^k}\prod_{i=0}^{n-1}\left(1-\frac{a}{\lambda^{i+k}x}\right)\\
        &=\frac{1}{\lambda^{\binom{k}{2}}x^k}\prod_{i=k}^{n+k-1}\left(1-\frac{a}{\lambda^{i}x}\right)\\
        &=\frac{1}{\lambda^{\binom{k}{2}}x^k}\frac{(x-a)_{\lambda}^{n+k}}{(x-a)_{\lambda}^k}.
    \end{align*}
\end{proof}

\begin{proposition}
    \begin{equation}
        \I_{\lambda}^{k}(x-a)_{\lambda}^n=
        \begin{cases}
            \frac{x^k}{\lambda^{\binom{k+1}{2}}},&\text{ if }n=0;\\
            \frac{x^k}{\lambda^{\binom{k+1}{2}}}(\lambda^{-k}x-a)_{\lambda}^{k}(x-a)_{\lambda}^{n-k},&\text{ if }n\neq1,\ k\leq n.
        \end{cases}
    \end{equation}
\end{proposition}
\begin{proof}
From Proposition \ref{prop_nth-int},
\begin{equation*}
   \I_{\lambda}^{k}\{1\}=\frac{x^k}{\lambda^{\binom{k+1}{2}}} 
\end{equation*}
and
\begin{align*}
    \I_{\lambda}^{k}(x-a)_{\lambda}^n&=\frac{x^k}{\lambda^{\binom{k+1}{2}}}(x/\lambda^k-a)_{\lambda}^n\\
    &=\frac{x^k}{\lambda^{\binom{k+1}{2}}}\prod_{i=0}^{n-1}\left(1-\frac{a}{\lambda^{i-k}x}\right)\\
    &=\frac{x^k}{\lambda^{\binom{k+1}{2}}}\prod_{i=0}^{k-1}\left(1-\frac{a}{\lambda^{i-k}x}\right)\prod_{i=k}^{n-1}\left(1-\frac{a}{\lambda^{i-k}x}\right)\\
    &=\frac{x^k}{\lambda^{\binom{k+1}{2}}}(\lambda^{-k}x-a)_{\lambda}^{k}\prod_{i=0}^{n-k-1}\left(1-\frac{a}{\lambda^{i}x}\right)\\
    &=\frac{x^k}{\lambda^{\binom{k+1}{2}}}(\lambda^{-k}x-a)_{\lambda}^{k}(x-a)_{\lambda}^{n-k}.
\end{align*}
\end{proof}

\begin{proposition}
Let $m$ and $n$ be two non-negative integers. Then the following assertion is valid
    \begin{equation}
        (x-a)_{\lambda}^{n+m}=(x-a)_{\lambda}^{m}(\lambda^mx-a)_{\lambda}^n.
    \end{equation}
\end{proposition}

\begin{definition}
Let $n$ be a non-negative integer. Then we set the following definition
    \begin{equation}
        (x-a)_{\lambda}^{-n}=\frac{1}{(\lambda^{-n}x-a)_{\lambda}^{n}}.
    \end{equation}
\end{definition}

\begin{proposition}
For any integer $n$,
\begin{equation}
    \D_{\lambda}^k(x-a)_{\lambda}^{n}=\frac{(x-a)_{\lambda}^{n+k}}{\lambda^{\binom{k}{2}}x^k(x-a)_{\lambda}^{k}}.
\end{equation}
\end{proposition}
\begin{proof}
If $n=-n^\prime<0$, we have
    \begin{align*}
        \D_{\lambda}^k(x-a)_{\lambda}^{n}&=\D_{\lambda}^k\bigg\{\frac{1}{(\lambda^{-n^\prime}x-a)_{\lambda}^{n^\prime}}\bigg\}\\
        &=\frac{1}{\lambda^{\binom{k}{2}}x^k(\lambda^{k-n^\prime}x-a)_{\lambda}^{n^\prime}}\\
        &=\frac{1}{\lambda^{\binom{k}{2}}x^k(\lambda^{k-n^\prime}x-a)_{\lambda}^{n^\prime-k}(x-a)_{\lambda}^k}\\
        &=\frac{(x-a)_{\lambda}^{n+k}}{\lambda^{\binom{k}{2}}x^k(x-a)_{\lambda}^{k}}.
    \end{align*}
\end{proof}

\begin{proposition}\label{prop_intk_binom}
For any integer $n$ and for all $k\geq1$,
    \begin{equation}
        \I_{\lambda}^k(x-a)_{\lambda}^n=\frac{x^k(\lambda^{-k}x-a)_{\lambda}^{k}(x-a)_{\lambda}^{n-k}}{\lambda^{\binom{k+1}{2}}}.
    \end{equation}
\end{proposition}
\begin{proof}
If $n=-n^\prime<0$ and $k\leq n$, we have
    \begin{align*}
        \I_{\lambda}^k(x-a)_{\lambda}^{n}&=\I_{\lambda}^k\bigg\{\frac{1}{(\lambda^{-n^\prime}x-a)_{\lambda}^{n^\prime}}\bigg\}\\
        &=\frac{x^k}{\lambda^{\binom{k+1}{2}}(\lambda^{-n^\prime-k}x-a)_{\lambda}^{n^\prime}}\\
        &=\frac{x^k(\lambda^{-k}x-a)_{\lambda}^k}{\lambda^{\binom{k+1}{2}}(\lambda^{-n^\prime-k}x-a)_{\lambda}^{n^\prime}(\lambda^{-k}x-a)_{\lambda}^k}\\
        &=\frac{x^k(\lambda^{-k}x-a)_{\lambda}^{k}}{\lambda^{\binom{k+1}{2}}(\lambda^{-n^\prime-k}x-a)_{\lambda}^{n^\prime+k}}\\
        &=\frac{x^k(\lambda^{-k}x-a)_{\lambda}^{k}(x-a)_{\lambda}^{n-k}}{\lambda^{\binom{k+1}{2}}}.
    \end{align*}
If $k>n$, from Theorem \ref{prop_nth-int}
\begin{align*}
    \I_{\lambda}^k(x-a)_{\lambda}^n&=\frac{x^k}{\lambda^{\binom{k+1}{2}}}(\lambda^{-k}x-a)_{\lambda}^n\\
    &=\frac{x^{k}(\lambda^{-k}x-a)_{\lambda}^n(\lambda^{-k+n}x-a)_{\lambda}^{-n+k}}{\lambda^{\binom{k+1}{2}}(\lambda^{-k+n}x-a)_{\lambda}^{-n+k}}\\
    &=\frac{x^{k}(\lambda^{-k}x-a)_{\lambda}^k}{\lambda^{\binom{k+1}{2}}(\lambda^{-k+n}x-a)_{\lambda}^{-n+k}}\\
    &=\frac{x^{k}(\lambda^{-k}x-a)_{\lambda}^k(x-a)_{\lambda}^{n-k}}{\lambda^{\binom{k+1}{2}}}.
\end{align*}
\end{proof}

\begin{proposition}
    \begin{equation}
        \D_{\lambda}^k\frac{1}{(x-a)_{\lambda}^n}=\frac{(x-a)_{\lambda}^k}{\lambda^{\binom{k}{2}}x^k(x-a)_{\lambda}^{n+k}}.
    \end{equation}
\end{proposition}
\begin{proof}
    \begin{align*}
        \D_{\lambda}^k\frac{1}{(x-a)_{\lambda}^n}&=\D_{\lambda}^k\frac{1}{(\lambda^{-n}(\lambda^nx)-a)_{\lambda}^n}\\
        &=\D_{\lambda}^{k}(\lambda^nx-a)_{\lambda}^{-n}\\
        &=\frac{1}{\lambda^{\binom{k}{2}}x^k}(\lambda^{n+k}x-a)_{\lambda}^{-n}\\
        &=\frac{1}{\lambda^{\binom{k}{2}}x^k}\frac{1}{(\lambda^{k}x-a)_{\lambda}^n}\\
        &=\frac{1}{\lambda^{\binom{k}{2}}x^k}\frac{(x-a)_{\lambda}^k}{(x-a)_{\lambda}^k(\lambda^{k}x-a)_{\lambda}^n}\\
        &=\frac{(x-a)_{\lambda}^k}{\lambda^{\binom{k}{2}}x^k(x-a)_{\lambda}^{n+k}}.
    \end{align*}
\end{proof}

\section{General $\lambda$-binomial}\label{sec6}

\begin{definition}
For any complex number $\alpha$, define     
\begin{equation}
    (x-a)_{\lambda}^\alpha=\frac{(x-a)_{\lambda}^{\infty}}{(\lambda^{\alpha}x-a)_{\lambda}^\infty}.
\end{equation}
\end{definition}

\begin{proposition}
For any two complex numbers $\alpha$ and $\beta$, we have
    \begin{equation}
        (x-a)_{\lambda}^{\alpha}(\lambda^{\alpha}x-a)_{\lambda}^{\beta}=(x-a)_{\lambda}^{\alpha+\beta}.
    \end{equation}
\end{proposition}
\begin{proof}
The proposition follows directly from the definition, since
    \begin{equation*}
        \frac{(x-a)_{\lambda}^{\infty}}{(\lambda^{\alpha}x-a)_{\lambda}^\infty}\frac{(\lambda^\alpha x-a)_{\lambda}^{\infty}}{(\lambda^{\alpha+\beta}x-a)_{\lambda}^\infty}=\frac{(x-a)_{\lambda}^{\infty}}{(\lambda^{\alpha+\beta}x-a)_{\lambda}^\infty}.
    \end{equation*}
\end{proof}

\begin{proposition}
For any complex number $\alpha$, we have
\begin{equation}
    \D_{\lambda}^{k}(x-a)_{\lambda}^\alpha=\frac{(x-a)_{\lambda}^{\alpha+k}}{\lambda^{\binom{k}{2}}x^k(x-a)_{\lambda}^k}.
\end{equation}
\end{proposition}
\begin{proof}
By definition, we have
    \begin{align*}
        \D_{\lambda}^{k}(x-a)_{\lambda}^\alpha&=\frac{(\lambda^kx-a)_{\lambda}^\alpha}{\lambda^{\binom{k}{2}}x^k}\\
        &=\frac{(\lambda^kx-a)_{\lambda}^{\infty}}{\lambda^{\binom{k}{2}}x^{k}(\lambda^{\alpha+k}x-a)_{\lambda}^{\infty}}\\
        &=\frac{(x-a)_{\lambda}^k(\lambda^kx-a)_{\lambda}^{\infty}}{\lambda^{\binom{k}{2}}x^k(x-a)_{\lambda}^k(\lambda^{\alpha+k}x-a)_{\lambda}^\infty}\\
        &=\frac{(x-a)_{\lambda}^\infty}{\lambda^{\binom{k}{2}}x^k(x-a)_{\lambda}^k(\lambda^{\alpha+k}x-a)_{\lambda}^\infty}=\frac{(x-a)_{\lambda}^{\alpha+k}}{\lambda^{\binom{k}{2}}x^k(x-a)_{\lambda}^k}.
    \end{align*}
\end{proof}

\begin{proposition}
For any complex number $\alpha$, we have
\begin{equation}
    \D_{\lambda}^{k}\left\{\frac{1}{(x-a)_{\lambda}^{\alpha}}\right\}=\frac{(x-a)_{\lambda}^k}{\lambda^{\binom{k}{2}}x^k(x-a)_{\lambda}^{\alpha+k}}.
\end{equation}
\end{proposition}
\begin{proof}
    \begin{align*}
        \D_{\lambda}^{k}\left\{\frac{1}{(x-a)_{\lambda}^{\alpha}}\right\}&=\frac{1}{\lambda^{\binom{k}{2}}x^k(\lambda^kx-a)_{\lambda}^\alpha}\\
        &=\frac{(\lambda^{\alpha+k}x-a)_{\lambda}^{\infty}}{\lambda^{\binom{k}{2}}x^k(\lambda^kx-a)_{\lambda}^{\infty}}\\
        &=\frac{(x-a)_{\lambda}^k(\lambda^{\alpha+k}x-a)_{\lambda}^{\infty}}{\lambda^{\binom{k}{2}}x^k(x-a)_{\lambda}^k(\lambda^kx-a)_{\lambda}^{\infty}}\\
        &=\frac{(x-a)_{\lambda}^k(\lambda^{\alpha+k}x-a)_{\lambda}^{\infty}}{\lambda^{\binom{k}{2}}x^k(x-a)_{\lambda}^{\infty}}\\
        &=\frac{(x-a)_{\lambda}^k}{\lambda^{\binom{k}{2}}x^k(x-a)_{\lambda}^{\alpha+k}}.
    \end{align*}
\end{proof}

\section{A $\lambda$-Taylor formula}\label{sec7}

Some useful identities in this section are:
\begin{align}
    (\lambda^{-i}a-a)_{\lambda}^i&=(\lambda;\lambda)_{i},\\
    (a-a)_{\lambda}^i&=\frac{1}{(\lambda;\lambda)_{i}}.
\end{align}

\begin{theorem}
For any function $f(x^{-1})$, with $f(x)$ a polynomial of degree $n$, and a non-zero real number $a$, we have the following $\lambda$-Taylor expansion:
\begin{equation}
    f(x^{-1})=\sum_{k=0}^{n}c_{k}(x-a)_{\lambda}^{k},
\end{equation}
where
\begin{align*}
    c_{0}&=f(x),\\
    c_{1}&=\frac{\lambda(\I_{\lambda}f)(a)-af(a)}{a(1-\lambda)}\\
    &\vdots\\
    c_{n}&=\frac{\lambda^{\binom{n+1}{2}}(\I_{\lambda}^{n}f)(a)}{a^n(\lambda;\lambda)_{n}}-\sum_{i=1}^{n-1}\frac{c_{i}}{(\lambda;\lambda)_{n-i}}.
\end{align*}
\end{theorem}
\begin{proof}
Let $V$ be the space of polynomials of degree not larger than $N$, so that $\dim(V)=n+1$. The polynomials $P_{n}(x)=(ax;\lambda^{-1})_{n}$ constitute a basis for $V$; i.e., any polynomial $P(x)\in V$ may be expressed as
\begin{equation}
    P(x)=\sum_{k=0}^{n}c_{k}(ax;\lambda^{-1})_{k}
\end{equation}
for some unique constants $c_{k}$. Through the mapping $x\mapsto x^{-1}$ we define the function 
\begin{equation}\label{eqn_expan-f}
    f(x)=P(x^{-1})=\sum_{k=0}^{n}c_{k}(x-a)_{\lambda}^k.
\end{equation}
Putting $x=a$ one gets $c_{0}=f(a)$. Then, apply the linear operator $\I_{\lambda}$ to both sides of the Eq.(\ref{eqn_expan-f})
\begin{equation}
    (\I_{\lambda}f)(x)=c_{0}\frac{x}{\lambda}+\left(\frac{x}{\lambda}-a\right)\sum_{k=1}^{n}c_{k}(x-a)_{\lambda}^{k-1}.
\end{equation}
Again, putting $x=a$, we get
\begin{equation}\label{eqn_sys2}
    (\I_{\lambda}f)(a)=f(a)\frac{a}{\lambda}+\left(\frac{a}{\lambda}-a\right)c_{1}
\end{equation}
and
\begin{equation}
    c_{1}=\frac{\lambda(\I_{\lambda}f)(a)-af(a)}{a(1-\lambda)}.    
\end{equation}
Once again, we apply the linear operator $\I_{\lambda}$ $i$ times, $2\leq i\leq n$, to both sides of the Eq.(\ref{eqn_expan-f}) and from Proposition \ref{prop_intk_binom},
\begin{align*}
    (\I_{\lambda}^if)(x)&=c_{0}\frac{x^i}{\lambda^{\binom{i+1}{2}}}+\sum_{k=1}^{n}c_{k}\frac{x^i(\lambda^{-i}x-a)_{\lambda}^i(x-a)_{\lambda}^{k-i}}{\lambda^{\binom{i+1}{2}}}\\
    &=f(a)\frac{x^i}{\lambda^{\binom{i+1}{2}}}+\frac{x^i(\lambda^{-i}x-a)_{\lambda}^i}{\lambda^{\binom{i+1}{2}}}\sum_{k=1}^{n}c_{k}(x-a)_{\lambda}^{k-i}.
\end{align*}
Putting $x=a$,
\begin{align}
    (\I_{\lambda}^if)(a)&=f(a)\frac{a^i}{\lambda^{\binom{i+1}{2}}}+\frac{a^i(\lambda^{-i}a-a)_{\lambda}^i}{\lambda^{\binom{i+1}{2}}}\sum_{k=1}^{i}c_{k}(a-a)_{\lambda}^{k-i}\nonumber\\
    &=f(a)\frac{a^i}{\lambda^{\binom{i+1}{2}}}+\frac{a^i(\lambda;\lambda)_{i}}{\lambda^{\binom{i+1}{2}}}\sum_{k=1}^{i}c_{k}\frac{1}{(\lambda;\lambda)_{i-k}}.\label{eqn_sys3}
\end{align}
Eqs. (\ref{eqn_sys2}) and (\ref{eqn_sys3}) define the system of linear equations
\begin{align}\label{eqn_system}
\begin{array}{ccccccccc}
    c_{0}& & & & & & &=&f(a)\\
    \frac{a}{\lambda}c_{0}& + &\frac{a(1-\lambda)}{\lambda}c_{1} & & & & & =&(\I_{\lambda}f)(a)\\
    \vdots& & \vdots &  & \ddots & & &\vdots\\
    \frac{a^n}{\lambda^{\binom{n+1}{2}}}c_{0}& + &\frac{a^n(\lambda;\lambda)_{n}}{\lambda^{\binom{n+1}{2}}(\lambda;\lambda)_{n-1}}c_{1} & + & \cdots &  + &\frac{a^n(\lambda;\lambda)_{n}}{\lambda^{\binom{n+1}{2}}}c_{n} &=&(\I_{\lambda}^nf)(a)
\end{array}.
\end{align}
By forward substitution, the solution of the system Eq.(\ref{eqn_system}) is
\begin{align*}
    c_{0}&=f(x),\\
    c_{1}&=\frac{\lambda(\I_{\lambda}f)(a)-af(a)}{a(1-\lambda)}\\
    &\vdots\\
    c_{n}&=\frac{\lambda^{\binom{n+1}{2}}(\I_{\lambda}^{n}f)(a)}{a^n(\lambda;\lambda)_{n}}-\sum_{i=1}^{n-1}\frac{c_{i}}{(\lambda;\lambda)_{n-i}}.
\end{align*}
The proof is reached.
\end{proof}

\begin{theorem}
The following connection formula holds.
    \begin{align*}
        \frac{1}{x^n}=\frac{1}{a^n}\sum_{k=0}^n(-1)^k\lambda^{\binom{k}{2}}\qbinom{n}{k}_{\lambda}(x-a)_{\lambda}^k.
    \end{align*}
\end{theorem}
From the above theorem, we have the following result.
\begin{theorem}
    \begin{equation}\label{eqn_reci_func}
        c_{0}+\frac{c_{1}}{x}+\cdots+\frac{c_{m}}{x^m}=\sum_{k=0}^{m}\left(\sum_{i=k}^{m}\frac{c_{i}}{a^i}\qbinom{i}{k}_{\lambda}\right)(-1)^k\lambda^{\binom{k}{2}}(x-a)_{\lambda}^k.
    \end{equation}
\end{theorem}

\begin{theorem}
The following connection formula holds.
    \begin{equation}
        (x-b)_{\lambda}^n=\sum_{k=0}^{n}\qbinom{n}{k}_{\lambda}(-1)^k\lambda^{2\binom{k}{2}}(\lambda^{1-n}b/a)^k(a-b)_{\lambda}^{n-k}(x-a)_{\lambda}^k.
    \end{equation}
\end{theorem}
\begin{proof}
From Eq.(\ref{eqn_gauss}), 
\begin{equation}
    (x-b)_{\lambda}^n=\sum_{k=0}^{n}\qbinom{n}{k}_{1/\lambda}\lambda^{-\binom{k}{2}}(b/x)^k
\end{equation}
and from Eq.(\ref{eqn_reci_func}), 
\begin{align*}
    &(x-b)_{\lambda}^n\\
    &=\sum_{k=0}^{n}\left(\sum_{i=k}^{n}\qbinom{n}{i}_{1/\lambda}\qbinom{i}{k}_{\lambda}\lambda^{-\binom{i}{2}}(b/a)^i\right)(-1)^{k}\lambda^{\binom{k}{2}}(x-a)_{\lambda}^k\\
    &=\sum_{k=0}^{n}\left(\sum_{i=k}^{n}\lambda^{i(i-n)}\qbinom{n}{i}_{\lambda}\qbinom{i}{k}_{\lambda}\lambda^{-\binom{i}{2}}(b/a)^i\right)(-1)^{k}\lambda^{\binom{k}{2}}(x-a)_{\lambda}^k\\
    &=\sum_{k=0}^{n}\frac{(\lambda;\lambda)_{n}}{(\lambda;\lambda)_{n}}\left(\sum_{i=k}^{n}\lambda^{\binom{i}{2}+i(1-n)}\frac{1}{(\lambda;\lambda)_{n-i}(\lambda;\lambda)_{i-k}}(b/a)^i\right)(-1)^{k}\lambda^{\binom{k}{2}}(x-a)_{\lambda}^k\\
    &=\sum_{k=0}^{n}\frac{(\lambda;\lambda)_{n}}{(\lambda;\lambda)_{n}}\left(\sum_{i=0}^{n-k}\lambda^{\binom{i+k}{2}+(i+k)(1-n)}\frac{1}{(\lambda;\lambda)_{n-k-i}(\lambda;\lambda)_{i}}(b/a)^{i+k}\right)(-1)^{k}\lambda^{\binom{k}{2}}(x-a)_{\lambda}^k\\
    &=\sum_{k=0}^{n}\lambda^{\binom{k}{2}}(\lambda^{1-n}b/a)^k\qbinom{n}{k}_{\lambda}\left(\sum_{i=0}^{n-k}\lambda^{\binom{i}{2}}\qbinom{n-k}{i}_{\lambda}(\lambda^{1-n+k}b/a)^{i}\right)(-1)^{k}\lambda^{\binom{k}{2}}(x-a)_{\lambda}^k\\
    &=\sum_{k=0}^{n}\qbinom{n}{k}_{\lambda}(-1)^{k}\lambda^{2\binom{k}{2}}(\lambda^{1-n}b/a)^k(a-b)_{\lambda}^{n-k}(x-a)_{\lambda}^k.
\end{align*}
The proof is reached.
\end{proof}
Connection formulas for the polynomials Eqs. (\ref{eqn_gauss}), (\ref{eqn_rs}), and (\ref{eqn_sw}) are:
\begin{equation}
    (x^{-1};\lambda)_{n}=\sum_{k=0}^{n}\qbinom{n}{k}_{\lambda}\lambda^{2\binom{k}{2}}a^{-k}(\lambda^{k}a^{-1};\lambda)_{n-k}(x-a)_{\lambda}^k,
\end{equation}
\begin{equation}
    \h_{n}(x^{-1}|\lambda)=\sum_{k=0}^{n}\qbinom{n}{k}_{\lambda}(-1)^k\lambda^{\binom{k}{2}}a^{-k}\h_{n-k}(a^{-1}|\lambda)(x-a)_{\lambda}^k,    
\end{equation}
and
\begin{equation}
    \s_{n}(x^{-1}|\lambda)=\sum_{k=0}^{n}\qbinom{n}{k}_{\lambda}(-1)^k\lambda^{\frac{k(3k-1)}{2}}a^{-k}\s_{n-k}(\lambda^2a^{-1}|\lambda)(x-a)_{\lambda}^k.    
\end{equation}

\end{document}